\documentclass[11pt]{amsart}
\usepackage{geometry}   
\usepackage{graphicx}
\usepackage{amssymb,amsthm}
\usepackage{mathscinet}
\usepackage{comment}
\usepackage{tikz}
\usepackage{comment}
\usepackage[foot]{amsaddr}

\renewcommand{\vec}{\mathbf}

\newtheorem{theorem}{Theorem}
\newtheorem{lemma}[theorem]{Lemma}

\theoremstyle{definition}
\newtheorem{remark}[theorem]{Remark}

\def\zbar{\bar z}

\def\lambdabar{\bar\lambda}
\def\bflambda{\boldsymbol{\lambda}}
\def\bfmu{\boldsymbol{\mu}}
\def\ii{\mathrm{i}}
\DeclareMathOperator{\Tr}{Tr}

\DeclareMathOperator{\sect}{sect}
\def\Nset{\mathbb{N}}
\def\Rset{\mathbb{R}}
\def\Cset{\mathbb{C}}
\def\calI{\mathcal I}
\let\epsilon=\varepsilon

\newcommand\Zbar{\overline Z}
\newcommand{\cupdot}{\mathbin{\mathaccent\cdot\cup}}

\newcommand\wt{\mathrm{wt}}
\newcommand\vol{\mathop{\mathrm{vol}}}
\renewcommand\Im{\mathop\mathrm{Im}\nolimits}
\renewcommand\Re{\mathop\mathrm{Re}\nolimits}

\tikzstyle{vertex} = [draw, shape=circle,minimum size=1mm,  inner sep=1pt, fill]

\title[A multivariate zero-free region\ldots]{Zero-free regions for the independence polynomial on restricted graph classes}
\author{Mark Jerrum}
\email{m.jerrum@qmul.ac.uk}
\author{Viresh Patel}
\email{v.patel@qmul.ac.uk}
\address{School of Mathematical Sciences, Queen Mary, University of London, Mile End Road, London E1~4NS.}

\date{\today}                                        

\begin{document}

\subjclass{05C31, 05C69, 05C75, 82B20}

\begin{abstract}
Generalising the Heilman-Lieb Theorem from statistical physics,
Chudnovsky and Seymour [\textit{J.\ Combin.\ Theory Ser.~B}, 97(3):350--357] showed that the univariate independence polynomial of any claw-free graph has all of its zeros on the negative real line. In this paper, we show that for any fixed subdivided claw $H$ and any $\Delta$, there is an open set $F \subseteq \mathbb{C}$ containing $[0, \infty)$ such that the independence polynomial of any $H$-free graph of maximum degree $\Delta$ has all of its zeros outside of $F$. We also show that no such result can hold when $H$ is any graph other than a subdivided claw or if we drop the maximum degree condition.

We also establish zero-free regions for the multivariate independence polynomial of $H$-free graphs of bounded degree when $H$ is a subdivided claw. The statements of these results are more subtle, but are again best possible in various senses.
\end{abstract}

\maketitle  

\section{Introduction}


The hard-core model, a model for the behaviour of gases in equilibrium, is intensely studied across various scientific disciplines, including statistical physics, probability, combinatorics, and theoretical computer science. In the latter two fields, the model is described using the language of independent sets in graphs --- the perspective adopted here. Central to this study is the model's partition function, also known as the independence polynomial, which captures crucial information about both the model and its underlying graph.

The (multivariate) hard-core model is specified by an underlying graph $G=(V,E)$ together with (complex) weights called \emph{fugacities} on the vertices $\boldsymbol{\lambda} = (\lambda_v)_{v \in V} \in \mathbb{C}^V$. For each subset of vertices $S \subseteq V$, we write $\lambda_{S} = \prod_{v \in S}\lambda_v$ and we write $\calI(G)$ for the set of all independent sets in~$G$ (i.e., for $S \subseteq V$, $S \in \calI(G)$ if and only if $uv \notin E$ for all $u,v \in S$).  Then the \emph{partition function} of the hard-core model with respect to $G$ and  $(\lambda_v)_{v \in V}$ is given by
$$Z_{G,\boldsymbol{\lambda}} = \sum_{S\in\calI(G)} \lambda_S.$$ 
We can equivalently think of $Z_{G,\boldsymbol{\lambda}}$ as a multivariate polynomial in the (complex) variables $(\lambda_v)_{v \in V}$ called the multivariate independence polynomial of $G$ and if we set all variables equal so that $\lambda_v = \lambda$ for all $v$, we obtain the univariate independence polynomial $Z_{G,\lambda}$.  


This paper presents generalisations of the classical Heilmann-Lieb Theorem \cite{HeilmannLieb}, a foundational result in statistical physics concerning the (absence of) zeros of the partition function of the monomer-dimer model. The monomer-dimer model is typically described in terms of matchings in graphs, but for our purposes, we describe it here equivalently as a special case of the hard-core model. Given a graph $H= (U,F)$, the line graph of $H$ is denoted $L(H)$ and is defined by $L(H) = (U',F')$ where $U' = F$ and $e,e' \in U' = F$ are adjacent in $L(H)$ if $e, e'$ are incident in $H$.  Note that there is a natural correspondence between matchings in~$H$ and independent sets in $L(H)$.  The Heilmann-Lieb Theorem significantly restricts the zeros of the univariate independence polynomial of line graphs.

\begin{theorem}[Heilman-Lieb \cite{HeilmannLieb}]
    If $G$ is a line graph (i.e., $G = L(H)$ for some $H$) then $Z_{G, \lambda} \not= 0$ for all $\lambda \in \mathbb{C} \setminus (- \infty, 0)$, i.e., all roots of $Z_{G, \lambda}$ are negative reals.
    \end{theorem}

It is a well-known fact (and easy to check) that if $G$ is a line graph, then it is claw free, that is, $G$ does not contain the claw (the first graph depicted in Figure~\ref{fig:stars}) as an induced subgraph. The class of claw-free graphs is considerably richer than the class of line graphs: indeed there has been considerable effort (see e.g.\ \cite{ChudSeyStructure}) to give a constructive description of claw-free graphs from line graphs.
Chudnovsky and Seymour \cite{ChudnovskySeymour} generalised the Heilmann-Lieb Theorem to claw-free graphs.

\begin{figure}
\begin{tikzpicture}[scale=0.7, semithick]
   \draw (-0.5,-0.87) node [vertex] (a) {};
   \draw (-0.5,0.87) node [vertex] (b) {};
   \draw (0,0) node [vertex] (c) {};
   \draw (1,0) node [vertex] (d) {};
   \draw (a) -- (c) -- (d);
   \draw (b) -- (c);
   
   \draw (2,-0.87) node [vertex] (fa) {};
   \draw (2,0.87) node [vertex] (fb) {};
   \draw (2.5,0) node [vertex] (fc) {};
   \draw (3.5,0) node [vertex] (fd) {};
   \draw (4.5,0) node [vertex] (fe) {};
   \draw (fa) -- (fc) -- (fd) -- (fe);
   \draw (fb) -- (fc);

  \draw (6.5,0) node [vertex] (ea) {};
   \draw (5.5,1) node [vertex] (eb) {};
   \draw (5.5,0) node [vertex] (ec) {};
   \draw (5.5,-1) node [vertex] (ed) {};
   \draw (6.5,-1) node [vertex] (ee) {};
   \draw (6.5,1) node [vertex] (ef) {};
   \draw (ea) -- (ec) -- (ed) -- (ee);
   \draw (ec) -- (eb) -- (ef);

\end{tikzpicture}
\caption{The claw, $S_{1,1,1}$; the fork, $S_{1,1,2}$; and the E, $S_{1,2,2}$.}
\label{fig:stars}
\end{figure}
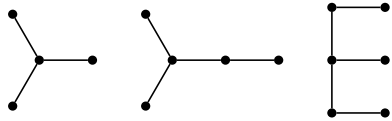
\begin{theorem}[Chudnovsky-Seymour \cite{ChudnovskySeymour}]
    If $G$ is a claw-free graph then $Z_{G, \lambda} \not= 0$ for all $\lambda \in \mathbb{C} \setminus (- \infty, 0)$, i.e., all roots of $Z_{G, \lambda}$ are negative reals.
\end{theorem}

The results above restrict the location of zeros in a very strong way showing an absence of zeros on almost the entire complex plane. An important feature of these zero-freeness results is the fact that zeros are restricted away from the positive real line 
since an absence of zeros around the physically relevant positive real values implies an absence of various types of phase transition (discontinuities in thermodynamic quantities) in the underlying model.
In this direction, we completely determine to what extent the Heilman-Lieb Theorem and Chudnovsky-Seymour Theorem above can be generalised to the setting of $H$-free graphs for different choices of $H$. 

The subdivided claw $S_{i,j,k}$ is the graph obtained from the claw by subdividing the three edges $i-1$, $j-1$, and $k-1$ times respectively, so in particular $S_{1,1,1}$ denotes the claw (see Figure~\ref{fig:stars} for examples). In naming the subdivided claws, we use the convention $i\leq j\leq k$. These graphs play a crucial role in the study of independent sets in graph theory~\cite{AbrishamiEtAl,ChudnovskyQuasi} and also in our results. Note that if $G$ is an $S_{i,j,k}$-free graph (by which we mean that $G$ does not contain an induced copy of $S_{i,j,k}$), then $G$ is $S_{t,t,t}$-free for any $t \geq \max(i,j,k)$.  
We show that there is a zero-free region containing $[0, \infty)$ for the class of bounded-degree $S_{i,j,k}$-free graphs (for each fixed choice of $i,j,k$).

\begin{theorem}
\label{thm:univariate}
    Given integers $t\geq1$ and $\Delta\geq3$, there exists an open region~$F \subseteq \mathbb{C}$ containing the positive reals $[0, \infty)$ such that if $G$ is a $S_{t,t,t}$-free graph of maximum degree~$\Delta$ and $\lambda \in F$ then $Z_{G, \lambda} \not= 0$. 
\end{theorem}

Theorem~\ref{thm:univariate} is best possible in two important ways. Firstly, Theorem~\ref{thm:univariate} does not hold if we replace $S_{i,j,k}$ with any graph~$H$ that is not a subdivided claw. More precisely, in Section~\ref{sec:notclaw}, for each $H$ that is not a subdivided claw (or a path, which might be considered as a degenerate subdivided claw), we exhibit a class of $H$-free graphs of maximum degree~3 for which zeros accumulate at some point on the positive real axis. 
Secondly, Theorem~\ref{thm:univariate} does not hold if we drop the condition that $G$ has maximum degree~$\Delta$. Indeed, in Section~\ref{sec:examples} we exhibit a class of graphs that are $S_{i,j,k}$-free for all $i,j \geq 1$ and $k \geq 2$ (but of unbounded degree) and whose zeros are dense on almost all of $\mathbb{C}$ including on the positive real line.
The only previous result in the direction of Theorem~\ref{thm:univariate} that we are aware of is due to Bencs~\cite{BencsPhD}, who showed that the (univariate) independence polynomial of $S_{1,1,2}$-free graphs of maximum degree $\Delta$ have a zero-free sector given by $|\arg(z)| \leq \pi / \Delta$.

Theorem~\ref{thm:univariate} is in fact derived from a stronger result for the multivariate independence polynomial. The multivariate setting is more subtle even for line graphs and claw-free graphs. For line graphs, Lebowitz, Ruelle, and Speer~\cite{LebowitzRuelleSpeer} proved Theorem~\ref{thm:LRSmultivariate} below using the method of Asano contractions.\footnote{As Theorem~\ref{thm:LRSmultivariate} is only implicit in~\cite{LebowitzRuelleSpeer}, we provide some details of the proof, including an explicit description and derivation of the zero-free set, in the Appendix (Section~\ref{sec:appendix}).} Let $R(k) \subseteq \mathbb{C}$ be the parabolic region of the complex plane defined by
\[
R(k) = \Big\{z \in \mathbb{C}: \Im(z)^2 < \frac{1}{k}\Re(z) + \frac{1}{4k^2}\Big\}.
\] 
\begin{theorem}[Lebowitz-Ruelle-Speer~\cite{LebowitzRuelleSpeer}]
\label{thm:LRSmultivariate}
    Suppose $G=(V,E)$ is a line graph whose largest clique has size at most $k$ (so the maximum degree of $G$ is at most $2k-1$). If $\boldsymbol{\lambda} \in R(k^2/4)^V \subseteq \mathbb{C}^V$ then $Z_{G,\boldsymbol{\lambda}} \not= 0$.  The same holds for line graphs of multigraphs (graphs with parallel edges).
\end{theorem}


Notice that Theorem~\ref{thm:LRSmultivariate} gives an open multivariate zero-free region containing the positive real line $[0, \infty)$ for line graphs of bounded degree. We give examples in Section~\ref{sec:examples} that show it is not possible to obtain such multivariate zero-free regions for $S_{i,j,k}$-free graphs when $j,k \geq 2$; however, we come close by giving open multivariate zero-free regions that contain any finite subinterval of $[0, \infty)$. 

\begin{theorem}\label{thm:zerofreeinterval}
Given $\lambda_0>0$ and integers $t\geq1$ and $\Delta\geq3$, there exists an open region~$F \subseteq \mathbb{C}$ containing the interval $[0,\lambda_0]$ such that if $G=(V,E)$ is a $S_{t,t,t}$-free graph of maximum degree~$\Delta$ and $\boldsymbol{\lambda} \in F^V$ then $Z_{G, \boldsymbol{\lambda}} \not= 0$. 
\end{theorem}

The proof can be found in Section~\ref{sec:zerofreeregion}. Of course, the examples for the univariate setting mentioned earlier also show that Theorem~\ref{thm:zerofreeinterval} cannot hold  if we replace $S_{i,j,k}$ with any graph $H$ that is not a subdivided claw or if we drop the maximum degree condition.

For $S_{1,1,k}$-free graphs, our examples in Section~\ref{sec:examples} do not preclude the possibility of an open multivariate zero-free region containing the entire positive real line $[0, \infty)$. While we cannot prove the existence of such a region for claw-free graphs, we obtain a result that extends Theorem~\ref{thm:LRSmultivariate} to a larger class of graphs (that almost includes all claw-free graphs) and that also enlarges the multivariate zero-free region in Theorem~\ref{thm:LRSmultivariate}. We need some notation to describe the result.

A simplicial clique $K$ in a graph $G$ is a clique such that for every vertex $v \in V(K)$, the neighbours of $v$ outside $K$ induce a clique in $G$, that is $G[\partial v \setminus V(K)]$ is a clique for all $v \in V(K)$. Write $\mathcal{G}^{\text{Cl-S}}$ for set of claw-free graphs for which every connected component has a simplicial clique and write $\mathcal{G}^{\textup{Cl-S}}_{k}$ for the set of graphs in $\mathcal{G}^{\textup{Cl-S}}$ whose largest clique has size at most $k$. The class $\mathcal{G}^{\textup{Cl-S}}$ has been previously studied, e.g., in \cite{CSSS, Freefermion,Amini}. It is also worth noting (see Lemma~\ref{lem:simplicial}) that 
given any connected claw-free graph of bounded degree say $\Delta$, one can delete at most $\Delta + 1$ vertices\footnote{Given any claw-free graph $G$, we claim that removing any vertex $v$ and all its neighbours gives a graph $G'$ in $\mathcal{G}^{\text{Cl-S}}$. Indeed, each component $C$ of $G'$ is adjacent to some neighbour $u$ of $v$ and it is easy to check that the neighbours of $u$ in $C$ form a simplicial clique (using that $G$ is claw-free).}  to obtain a graph in $\mathcal{G}^{\textup{Cl-S}}$; thus every claw-free graph of bounded degree is very close to a graph in $\mathcal{G}^{\textup{Cl-S}}$.
We prove the following in Section~\ref{sec:clawfree}.
\begin{theorem}
\label{thm:almostclawfree}
    For any graph $G \in \mathcal{G}^{\textup{Cl-S}}_{k}$ if $\boldsymbol{\lambda} \in R(k)^V \subseteq \mathbb{C}^V$ then $Z_{G,\boldsymbol{\lambda}} \not= 0$. 
\end{theorem}
Note that any line graph whose largest clique has size at most $k$ belongs to $\mathcal{G}^{\textup{Cl-S}}_{k}$ and so Theorem~\ref{thm:almostclawfree} extends Theorem~\ref{thm:LRSmultivariate} to a larger graph class.
Furthermore Theorem~\ref{thm:almostclawfree} gives a multivariate zero-free region $R(k)$ that matches the $R(k^2/4)$ in Theorem~\ref{thm:LRSmultivariate} for $k=4$, and is larger for $k>4$. We give examples in Section~\ref{sec:examples} showing that essentially any multivariate zero-free region for $\mathcal{G}^{\textup{Cl-S}}_k$ must be contained in $R(k/2)$, so the multivariate zero-free region we give cannot be enlarged by much (if at all).

We point out that from Theorem~\ref{thm:univariate} it follows using Barvinok's interpolation method~\cite{Barvinok} together with the results in \cite{PatelRegts} that there is a deterministic polynomial-time approximation algorithm (specifically a fully polynomial-time approximation scheme) to compute $Z_{G, \lambda}$ for any $S_{i,j,k}$-free graph of maximum degree $\Delta$ and any $\lambda \in F$ with $F$ as given in Theorem~\ref{thm:univariate}. Similar algorithms follow from Theorem~\ref{thm:zerofreeinterval} and Theorem~\ref{thm:almostclawfree}. Previously, the second author~\cite{Jerrum2024} gave a randomised approximation algorithm for the same computation for positive real values of $\lambda$ (also in the multivariate setting). This all builds on previous work~\cite{JerrumSinclair, Matthews, ChenGu} that we do not review here.


\subsubsection*{{\bf Related results}}
 Leake and Ryder~\cite{LeakeRyder} consider a weaker form of multivariate zero-freeness called same-phase stability (in which all variables have the same argument) and establish that same-phase stability holds for independence polynomials of all claw-free graphs. They also establish bounds to the closest (univariate) zero for claw-free graphs and graphs in $\mathcal{G}^{\textup{Cl-S}}_{k}$,
which Fialho and Procacci~\cite{fialho2025} recently improved on (in certain settings).

Besides $H$-free graphs,
there is also considerable interest in the hardcore model on lattices since they often arise in nature. Usually, there is no open zero-free region containing the entire positive real line, but one seeks to bound the zeros in the complex plane e.g.\ for the square lattices of fixed dimension (and their subgraphs)~\cite{deBoer2024, Helmuth2020} and hierarchical lattices~\cite{HierLattice2025}. Lattices such as the kagome lattice or the pyrochlore lattice are in fact line graphs so that the Heilman-Lieb Theorem already applies. 

For certain classes of trees, Bencs~\cite{Bencs2018} showed real rootedness of the independence polynomial.
The class of general trees of bounded degree has received considerable attention because, as it turns out, the zeros of trees of maximum degree $\Delta$ coincide exactly with zeros of general graphs of maximum degree $\Delta$ (see e.g.\ \cite{Bencs2018}).
In this direction, Dobrushin~\cite{Dobrushin1, Dobrushin2} and Shearer~\cite{Shearer} independently established a zero-free disc for graphs of maximum degree $\Delta$ that is  centered at zero of radius $\lambda_s(\Delta) = (\Delta - 1)^{\Delta -1}/ \Delta^{\Delta}$.
While this zero-free disc cannot be enlarged in the negative direction, in the positive real direction, Peters and Regts~\cite{PetersRegts} found a zero-free region that contains a neighbourhood of the interval $[0, \lambda_c(\Delta))$, where $\lambda_c(\Delta) = (\Delta - 1)^{\Delta -1} / (\Delta - 2)^{\Delta} > \lambda_s(\Delta)$ thus resolving a conjecture of Sokal~\cite{Sokal}.
This began a sequence of works which together show that,
for graphs of maximum degree $\Delta$, 
there is a cardioid region $C_{\Delta} \subseteq \mathbb{C}$ (that intersects $\mathbb{R}$ in the interval ($-\lambda_s(\Delta)$, $\lambda_c(\Delta)$))  outside of which  zeros of the independence polynomial are dense \cite{deBoerBGPR, BezakovaGGS} but inside of which there are both zeros~\cite{Buys} and zero-free regions~\cite{BCSVondrak} suggesting an exact description might be difficult to give (although there is an understanding of the limiting behaviour as $\Delta \to \infty$~\cite{BBP25}).

\section{Notation}

For a graph $G=(V,E)$ and a vertex $v \in V$, we write $\partial v$ for the neighbours of $v$ in $G$.
For $U\subseteq V$, we write $\partial U=\{v\in V\setminus U:v \in \partial u \text{ for some }u\in U\} = (\cup_{u \in U} \partial u) \setminus U$ for the boundary of $U$. We write $G[U]$ for the graph induced by $G$ on~$U$. We write $U^{(\leq r)} := \{S \subseteq U: 1 \leq |S| \leq r\}$.

We recall and extend some notation from the Introduction.
Given complex vertex weights $\boldsymbol{\lambda}=(\lambda_v)_{v\in V} \in \mathbb{C}^V$ and $S \subseteq V$, we write $\lambda_{S} = \prod_{v \in S}\lambda_v$.  We write $\calI(G)$ for the set of all independent sets in~$G$.  Then the \emph{(multivariate) independence polynomial} of $G$ is 
$$Z_{G,\boldsymbol{\lambda}} = \sum_{S\in\calI(G)} \lambda_S.$$
As the graph $G$ remains fixed most of the time, we generally use the abbreviated forms $\calI(U)=\calI(G[U])$ and $Z(U) = Z_{G[U],\boldsymbol{\lambda}[U]}$, where $\boldsymbol{\lambda}[U]$ are the weights inherited from~$G$ by~$G[U]$.


Recall that for $1\leq i\leq j\leq k$, we write $S_{i,j,k}$ for the subdivided claw whose leaves are at distance $i$, $j$ and $k$ from the degree-3 vertex (see Figure~\ref{fig:stars}). 
  Note that the class of $S_{i,j,k}$-free graphs is contained in the class of $S_{t,t,t}$-free graphs, where $t=\max\{i,j,k\}$, so we lose little in what follows by focusing on symmetric subdivided claws.

\section{Zero-free regions for graphs with no induced subdivided claw}
\label{sec:zerofreeregion}

Given a graph $G=(V,E)$ together with a fixed root vertex $u \in V$,
we define admissible pairs of vertex subsets recursively as follows. For $L, U \subseteq V$, we define $(L,U)$ to be \emph{admissible} (with respect to $G, u$) if $(L,U)$ satisfies
\begin{itemize}
    \item $(L,U) = (\emptyset, V)$; or
    \item $(L,U) = (\{ u \}, V \setminus \{ u \})$; or
    \item $L$ is an independent set satisfying
    $\emptyset \not= L \in \calI(\partial L' \cap U')$ and  $U = U' \setminus  \partial L' $ for some admissible pair $(L',U')$. 
\end{itemize}
We use the unusual convention that $\partial \emptyset = \{ u \}$ in order to omit the second item above and streamline our inductive argument. 
It is easy to see that $U \cap L = \emptyset$ for any admissible pair $(L,U)$ and so $(L,V)$ is admissible if and only if $L = \emptyset$; this is a fact we use later. See Figure~\ref{fig:admissible} for an illustration of the inductive step (third bullet point above) in the definition of admissible pair.

 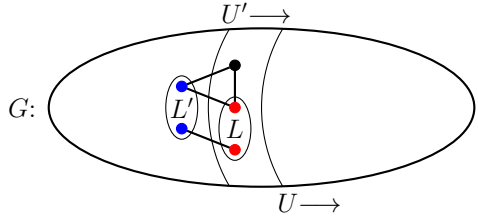
\begin{figure}
    \centering
      \begin{tikzpicture}[scale=0.7, inner sep=2pt, transform shape]
        \draw (0,0) ellipse (4 and 1.5) [thick];
        \begin{scope}
          \clip (0,0) ellipse (4 and 1.5);
          \draw (3,0) ellipse (3 and 3);
          \draw (2,0) ellipse (3 and 3);
        \end{scope}
        \draw (-0.5,-0.4) ellipse (0.3 and 0.6);
        \draw (-1.5,0) ellipse (0.3 and 0.6);
        \draw (-1.5,0) node {\Large $L'$};
        \draw (-0.5,-0.4) node {\Large $L$};
        \draw (-0.1,1.8) node {\Large $U'{\longrightarrow}$};
        \draw (+0.9,-1.8) node {\Large $U{\longrightarrow}$};
        \draw (-4.5,0) node {\Large $G$:};
        \draw (-1.5,0.4) node [draw,fill,shape=circle,color=blue] (u1) {};
        \draw (-1.5,-0.4) node [draw,fill,shape=circle,color=blue] (u2) {};
        \draw (-0.5,0.8) node [draw,fill,shape=circle,color=black] (v1) {};
        \draw (-0.5,0) node [draw,fill,shape=circle,color=red] (v2) {};
        \draw (-0.5,-0.8) node [draw,fill,shape=circle,color=red] (v3) {};
        \path (u1) edge [thick] (v1);
        \path (u1) edge [thick] (v2);
        \path (u2) edge [thick] (v3);
        \path (v1) edge [thick] (v2);
      \end{tikzpicture}  
    \caption{A derivation of admissible pair $(U,L)$ from admissible pair $(U',L')$.}
    \label{fig:admissible}
  \end{figure}

The introduction of admissible pairs is motivated by a process used in~\cite{Jerrum2024} to sample red-blue pairs of random independent sets in~$G$.  The union $R\cup B$ of independent sets $R$ and~$B$ induces connected components in the graph~$G$.  The inductive definition of admissible pairs given above mirrors the process of growing the red-blue connected component from a distinguished vertex~$u$, presented in \cite[Algorithm~1]{Jerrum2024}.  

Our approach to proving non-vanishing of the partition function is an induction over pairs $(L,U)$ of vertex subsets. The admissible pairs arise as the pairs $(L,U)$ that appear in that inductive argument.  The usage here may seem quite different to that in \cite{Jerrum2024}, but the crucial property of $S_{t,t,t}$-free graphs that we rely on is the same, namely that $|L|$ remains bounded.  (For general graphs with degree-bound $\Delta$ the set $L$ may grow unceasingly.) The following is a restatement of Lemma 14 of~\cite{Jerrum2024}.

\begin{lemma}
\label{lem:1dstructure}
    Let $G$ be a graph of maximum degree $\Delta$ that is $S_{t,t,t}$-free. If $(L,U)$ is an admissible pair then $|L| \leq 2 \vol(\Delta, 2t)$, where
    \[
     \vol(\Delta, t) := 1 + \Delta \frac{(\Delta - 1)^t - 1}{\Delta - 2}. 
    \]
\end{lemma}

Our goal is to show that $Z_{G,\boldsymbol{\lambda}} = Z(V)$ is non-zero for $S_{t,t,t}$-free graphs $G=(V,E)$ when the complex weights $\boldsymbol{\lambda}$ are close to the positive real line. The first lemma establishes invariant regions of the complex plane under complex transformations that arise as recursions in our induction proof.
Recall that we write $R(k) \subseteq \mathbb{C}$ for the parabolic region in the complex plane given by
\[\label{eq:parabola}
R(k) = \Big\{z \in \mathbb{C}: \Im(z)^2 < \frac{1}{k}\Re(z) + \frac{1}{4k^2}\Big\}.
\]
Similarly, write $H(t)$ for the halfplane in $\mathbb{C}$ given by 
\[
H(t) = \{z \in \mathbb{C}: \Re(z) \geq t\}.
\]

\begin{lemma}
\label{lem:invariant}
Fix $\vec{a} = (a_1, \ldots, a_k) \in R(k)^k$ and consider the
   complex function $f_{\vec{a}}: \mathbb{C}^k \rightarrow \mathbb{C}$ given by $f_{\vec{a}}(z_1, z_2, \ldots, z_k) = 1 + \sum_{j=1}^k a_jz_j^{-1}$. Then $f_{\vec{a}}(H(1/2)^k) \subseteq H(1/2)$. 
\end{lemma}
\begin{proof}
    Given $\vec{z} = (z_1, \ldots, z_k) \in H(1/2)^k$, we have for each $j$ that $z_j^{-1}$ lies inside the disc centered at $1$ of radius $1$ and so $a_jz_j^{-1}$ lies in the disc centered at $a_j$ of radius $|a_j|$. The leftmost point of this disc has real part $\Re(a_j) - |a_j|$ and furthermore, by our choice of $a_j$, we have
    \begin{align*}
    \Re(a_j) - |a_j| 
    &= \Re(a_j) - (\Re(a_j)^2 + \Im(a_j)^2)^{1/2} \\
    &\geq \Re(a_j) - \big(\Re(a_j)^2 + \Re(a_j)/k + (1/2k)^2\big)^{1/2} 
    = -1/2k.
    \end{align*}
    This shows that $\Re(a_jz_j^{-1}) \geq -1/2k$ and so $\Re(f_{\vec{a}}(z_1, \ldots, z_k)) \geq 1/2$, i.e., $f_{\vec{a}}(\vec{z}) \in H(1/2)$.
\end{proof}

\begin{lemma}
\label{lem:ind2}
Fix positive integers $t, \Delta$ and let $r = 2\vol(\Delta, 2t)$. 
    For any $S_{t,t,t}$-free graph $G=(V,E)$ of maximum degree $\Delta$, any complex vertex weights $\boldsymbol{\lambda}=(\lambda_v)_{v \in V}$ satisfying $\lambda_S \in R(2^{\Delta r})$ for all $S \in V^{(\leq r)}$, any root vertex $u \in V$, and any admissible pair $(L,U)$, we have
    \begin{itemize}
        \item[(i)] $Z(U), Z(U \setminus \partial L) \not= 0$;
        \item[(ii)] $Z(U) / Z(U \setminus \partial L)) \in H(1/2)$.
    \end{itemize}
    In particular $Z_{G,\boldsymbol{\lambda}} \not= 0$ by taking the admissible pair $(\emptyset, V)$.
\end{lemma}
\begin{proof}
    We prove the lemma by induction on $|U| + |V|$. For the base case, when $(L,U) = (L,\emptyset)$, we have $Z(U) = Z(U \setminus \partial L) = 1$ and (i) and (ii) immediately follow.

    For the induction step, suppose we are given an admissible pair $(L,U)$ (with respect to $G,u$) with $|U| \geq 1$. We have the following recursion:
    \begin{equation}
         Z(U) = Z(U \setminus \partial L) + \sum_{\emptyset \not= S \in \calI(U \cap \partial L)} \lambda_S Z(U \setminus (\partial L \cup \partial S)).
    \end{equation}  
Note that $Z(U \setminus \partial L) \not= 0$ by induction using (i) by considering the admissible pair $(\emptyset, U \setminus \partial L)$ with respect to the smaller\footnote{Note that $U \setminus \partial L$ must be a strict subset of $V$ since either $U$ is a strict subset of $V$ or $(L,U)=(\emptyset,V)$ (from earlier) in which case $\partial L =\{u\}$.} graph $G[U \setminus \partial L]$ and any root vertex $u' \in U \setminus \partial L$.    
Therefore, we can divide through by $Z(U \setminus \partial L)$ to obtain 
\begin{equation}
\label{eq:recursion}
         \frac{Z(U)}{Z(U \setminus \partial L)} = 1 + \sum_{\emptyset \not= S \in \calI(U \cap \partial L)} \lambda_S \left( \frac{Z(U \setminus \partial L)}{Z(U \setminus (\partial L \cup \partial S))} \right)^{-1}.
\end{equation}     

If $U \cap \partial L = \emptyset$ then $U = U \setminus \partial L$ and so $Z(U) = Z(U \setminus \partial L) \not= 0$ (as established above) from which (i) and (ii) follow. 
If $U \cap \partial L \not= \emptyset$, then for any nonempty $S\in\calI(U \cap \partial L)$, the pair $(S, U \setminus \partial L)$ is admissible (by definition). 
Therefore $Z(U \setminus (\partial L \cup \partial S)) \not= 0$ by induction (i) and  $\frac{Z(U \setminus \partial L)}{Z(U \setminus (\partial L \cup \partial S))} \in H(1/2)$ by induction (ii). 
Recall from Lemma~\ref{lem:1dstructure} that $|L| \leq r = 2\vol(\Delta, 2t)$ and so $|\partial L| \leq  \Delta r$; this means the sum above has at most $2^{\Delta r}$ terms.   
(This is a loose bound;  a better one would follow from noting that $|S|\leq r$).  By our choice of $\boldsymbol{\lambda}$ and noting $|S| \leq r$ (by Lemma~\ref{lem:1dstructure}), we have $\lambda_S \in R(2^{\Delta r})$. Applying Lemma~\ref{lem:invariant} to \eqref{eq:recursion} shows (ii). From this (i) follows.
\end{proof}

As noted in the Introduction, we cannot obtain a multivariate zero-free region containing $(0, \infty)$, but we can get such a univariate zero-free region, and also an open multivariate zero-free region containing any finite interval of the positive real axis. 

First, to obtain the univariate zero-free region, we specialise Lemma~\ref{lem:ind2} to the admissible pair $(\emptyset,V)$ and the constant weighting where $\lambda_v = \lambda$ for all $v \in V$.

\begin{proof}[Proof of Theorem~\ref{thm:univariate}]
We are given positive integers $t, \Delta$. Set $r = 2\vol(\Delta, 2t)$.
For each $\ell \in \mathbb{N}$, let $F_{\ell} = \{\lambda \in \mathbb{C} : \lambda^{\ell} \in R(2^{\Delta r})\}$, which is an open set containing $[0, \infty)$.
Let
\[
F= \{\lambda \in \mathbb{C} : \lambda^{\ell} \in R(2^{\Delta r}) \text{ \: for all \: } \ell = 1, \ldots, r\} = \bigcap_{\ell=1}^r F_{\ell},
\]
which is also an open set containing $[0, \infty)$.
For any $S_{t,t,t}$-free graph $G = (V,E)$ of maximum degree $\Delta$ and any $\lambda \in F$,
setting $\lambda_v = \lambda$ for all $v \in V$, we have $\lambda_S \in R(2^{\Delta r})$ for any $S \in V^{(\leq r)}$.
Thus we can apply Lemma~\ref{lem:ind2} to conclude $Z_{G, \lambda} \not= 0$. 
\end{proof}



\begin{remark}
 We derive a slightly smaller but more explicit zero-free region. For the set $F$ defined in the proof above,  we have
 \[
F \supseteq F^* = \left\{ (x+\ii y): |y| \leq Ax^{-\frac{r-2}{2}}, x \geq \Delta^{-1}  \right\}.
 \]
 for a suitably small choice of constant $A$ depending only on $r$ and $\Delta$ (recall that $r$ itself only depends on $\Delta$ and $t$).
To see this, we must show that for any $z = x+\ii y \in F^*$, we have $z^{\ell} \in R(2^{\Delta r})$ for each $\ell = 1, \ldots, r$.
Note first that
\[
|\Im(z^{\ell})| = |\Im((x+\ii y)^{\ell})| 
= \left| \sum_{t \text{ odd}}^{\ell} \binom{\ell}{t}(\ii y)^tx^{\ell-t} \right| 
\leq \sum_{t \text{ odd}}^{\ell} \binom{\ell}{t}|y|x^{\ell-1} \leq 2^{\ell}|y|x^{\ell-1},
\]
and similarly $|\Re(z^{\ell})| \geq x^{\ell} - 2^{\ell}|y|^2x^{\ell - 2} \geq \frac{1}{2}x^{\ell}$, where both inequalities assume that $A$ is sufficiently small  (take e.g.\ $A \leq 2^{-\frac{1}{2}(r+1)}\Delta^{-\frac{1}{2}r}$). Now one can verify that $\Im(z^{\ell})^2 \leq 2^{-\Delta r}\Re{z^{\ell}}$ for sufficiently small $A$ (take e.g.\ $A \leq 2^{-\frac{1}{2}(\Delta r + 2r + 1)}$) showing that $z^{\ell} \in R(2^{\Delta r})$. 

While $F^*$ does not contain the origin, we know the disc of radius $(\Delta - 1)^{\Delta - 1}/ \Delta^{\Delta} \approx e/ \Delta$ is zero-free (see the Introduction) and this disc intersects $F^*$.
\end{remark}





Finally, we come to the existence of multivariate zero-free regions including arbitrary finite intervals of the positive real axis. 

Given any region $R \subseteq \mathbb{C}$ and positive integer $p$, we define $R^{\times p} := \{\prod_{i=1}^{q} z_i :  1 \leq q \leq p \text{ and } z_1, \ldots, z_q \in R  \}$. Also, we write $S(\rho, \theta) := \{z:|z|< \rho \text{ and } |\arg z|<\theta\}$ for the bounded sector of radius $\rho$ and angle $\theta$ and we write $B(\rho):= \{z : |z| < \rho\}$ for the disc of radius $\rho$ centered at $0$.

\begin{proof}[Proof of Theorem~\ref{thm:zerofreeinterval}]

First note that for $\rho_1 \geq 1$, we have $S(\rho_1, \theta)^{\times p} = S(\rho_1^{p}, p \theta)$ and for $\rho_2 \leq 1$, we have $B(\rho_2)^{\times p} = B(\rho_2)$. Furthermore, again assuming $\rho_2 \leq 1 \leq \rho_1$, it is easily checked that $(S(\rho_1, \theta) \cup B(\rho_2))^{\times p} \subseteq S(\rho_1^{p}, p \theta) \cup B(\rho_1^{p-1} \rho_2)$.

In the statement of Theorem~\ref{thm:zerofreeinterval}, we are given $\lambda_0 > 0$ and positive integers $t, \Delta$. As before, set $r = 2\vol(\Delta, 2t)$ and without loss of generality, we may assume $\lambda_0\geq1$.  Fix $\epsilon>0$ and  set $\rho_1 = \lambda_0 + \epsilon$. From the discussion in the previous paragraph, we can choose $\theta, \rho_2>0$ sufficiently small such that the set 
$$F:= S(\rho_1, \theta) \cup B(\rho_2)$$
satisfies that $F^{\times r} \subseteq R(2^{\Delta r})$.  
Then, by Lemma~\ref{lem:ind2}, $F$ is a multivariate zero-free region for~$Z_{G,\boldsymbol{\lambda}}$ for any $S_{t,t,t}$-free graph of maximum degree $\Delta$. 
\end{proof}

\section{Claw-free graphs}\label{sec:clawfree}

We begin this section with some structural properties of claw-free graphs. 
Recall that a simplicial clique in a claw-free graph $G$ is a clique $K$ in $G$ such that for every vertex $v \in V(K)$,   $G[\partial v \setminus K]$ is a clique. Recall that $\mathcal{G}^{\textup{Cl-S}}$ is the set of claw-free graphs whose every component has a simplicial clique and $\mathcal{G}^{\textup{Cl-S}}_k$ is the set of graphs in $\mathcal{G}^{\textup{Cl-S}}$ whose largest clique has size at most $k$.
We will require the following structural lemma; similar results are mentioned in~\cite{ChudnovskySeymour, Freefermion}.

\begin{lemma}
\label{lem:simplicial}
    If $G \in \mathcal{G}^{\textup{Cl-S}}$  and $X \subseteq V(G)$ then $G' = G[V(G) \setminus X] \in \mathcal{G}^{\textup{Cl-S}}$,  i.e.\ $\mathcal{G}^{\textup{Cl-S}}$ is a hereditary graph class. 
\end{lemma}

\begin{proof}
It is sufficient to prove the lemma in the case that $X=\{v\}$ is a single vertex and $G$ is connected.
Let $K$ be a simplicial clique of $G$.
We distinguish two cases, depending on whether or not $G'$ is connected.

Suppose first that removing vertex $v$ separates $G$ into $k$~connected components $G_1,\ldots,\allowbreak G_k$, with $k\geq2$.  Then, for all $i\in\{1,\ldots,k\}$ the vertex subset $U=\partial v\cap V(G_i)$ induces a clique.  
(Suppose $u,u'\in U$ and $w\in \partial v \cap V(G_j)$ with $j \not= i$.  Then $uu'$ must be an edge, otherwise $\{v,u_,u',w\}$ would induce a claw.)  
Also, $U$ must be simplicial in $G_i$.  (Suppose $u\in U$ and $w,w'\in \partial u \cap(V(G_i)\setminus U)$.  Then $ww'$ is an edge, otherwise $\{u,v,w,w'\}$ would induce a claw.)

The other possibility is that $G'=G-v$ is connected.  If $v\notin K$ then the clique $K$ remains simplicial.  If $v\in K$ and $|K|\geq2$ then $K-v$ is a simplicial clique in~$G'$.  Finally, if $K=\{v\}$, then $U= \partial v$ is a clique (since $K$ is simplicial) and $U$ is simplicial in $G'$ (otherwise there would be vertices $u\in U$ and 
$w,w'\in \partial u \setminus (U\cup \{v\})$ such that $\{u,v,w,w'\}$ induces a claw.
\end{proof}



Next we require some notation on paths. 
 Let $G$ be a graph and $P$ be a path in $G$ with $|V(P)| \geq 1$. Let $u$ and $x$ be the endpoints of $P$ (where we allow $u=x$ when $P$ has a single vertex). When working with any path, we specify one of its endpoints to be \emph{active} and we indicate that $x$  is the active vertex by writing $Px$ (and in the case when $P$ has one vertex, that vertex is taken to be active). 
 We define
\[
A_{Px}(G) = V(P) \cup \partial(V(P) \setminus \{x\}) 
\]
and
\[
U_{Px}(G) = V(G) \setminus A_{Px}. 
\]
We simply write $A_{Px}$ and $U_{Px}$ when $G$ is clear from context.
Given $Px$ in $G$ and $y \in \partial x \cap U_{Px}$, then we write $Pxy$ for the path $P$ with vertex $y$ appended to $x$ and where $y$ is now the active vertex. All paths we consider have at least one vertex.

We give one further piece of notation on simplicial cliques. For a claw-free graph $G$ with a simplicial clique $K$, we write $G + u_K$ for the graph obtained from $G$ by adding a new vertex $u_K$ that is adjacent to every vertex in $K$. It is easy to check that $G + u_K$ is claw-free and we call $u_K$ the artificial simplicial vertex, which is a natural starting point for our paths in the next lemma.

\begin{lemma}\label{lem:almostclawfree}
    For any $G = (V,E) \in \mathcal{G}_k^{\textup{Cl-S}}$ 
    with a simplicial clique $K$, any path $Px$ in $G + u_K$ starting at the artificial simplicial vertex $u_K$ and ending at $x$ (where we allow $x=u_K$), and any complex vertex weights $\boldsymbol{\lambda}=(\lambda_v)_{v \in V} \in R(k)^V \subseteq \mathbb{C}^V$ we have 
    \begin{itemize}
        \item[(i)] $Z(U_{Px}), Z(U_{Px} \setminus \partial x) \not= 0$ and
        \item[(ii)] $ \Re (Z(U_{Px}) / Z(U_{Px} \setminus \partial x) ) \geq 1/2$.
    \end{itemize}    
\end{lemma}
The partition function $Z$ in the lemma above is with respect to the graph $G + u_K$, but note that since our path always includes the vertex $u_K$ as a starting point, then $U_{Px}$ always excludes $u_K$, so we do not need to specify a weight for $u_K$.
\begin{proof}[Proof of Lemma~\ref{lem:almostclawfree}]
    We do induction on $|U_{Px}| + |V|$. If $|U_{Px}|=0$ then $Z(U_{Px})= Z(U_{Px} \setminus \partial x) = 1$ and both (i) and (ii) immediately follow.

    For the induction step first consider the case $U_{Px} \cap \partial x = \emptyset$. We see that (ii) follows immediately provided (i) holds.
    For (i), note that $U_{Px} \cap \partial x = \emptyset$ can only happen if $|V(Px)| \geq 2$, in which case $G' = G[U_{Px}]$ is strictly smaller than $G$ and every component of $G'$ has a simplicial clique by Lemma~\ref{lem:simplicial}. By applying induction to each of those components, we have $Z(U_{Px}) \not= 0$  so (i) follows. 
 Thus we may assume $U_{Px} \cap \partial x \not= \emptyset$.

Now assuming $U_{Px} \cap \partial x \not= \emptyset$, let us check that $U_{Px} \cap \partial x$ induces a clique in $G$. If $|V(Px)| = 1$ then $P$ consists of the single simplicial vertex $x = u_K$, so $\partial x$ induces a clique. If $|V(Px)| \geq 2$ and $a,b \in U_{Px} \cap \partial x$ are non-adjacent then $\{ x^-, x, a, b \}$ induces a claw in $G$ where $x^-$ is the predecessor of $x$ on $P$ (a contradiction). Hence $U_{Px} \cap \partial x$ induces a clique (of size at most $k$), and we have
\[
Z(U_{Px}) = Z(U_{Px} \setminus \partial x) + \sum_{y \in U_{Px} \cap \partial x}\lambda_y Z(U_{Px} \setminus (\partial x \cup \partial y)).
\]    
Note that $U_{Px} \setminus \partial x = U_{Pxy}$ for any $y$ adjacent to $x$ and so by induction (using (i)), we have $Z(U_{Px} \setminus \partial x) = Z(U_{Pxy}) \not= 0$, so we can divide the recursion above by this quantity to obtain 
\begin{align*}
\frac{Z(U_{Px})}{Z(U_{Px} \setminus \partial x)} 
&= 1 + \sum_{y \in U_{Px} \cap \partial x}\lambda_y \left( \frac{Z(U_{Px} \setminus \partial x)}{Z(U_{Px} \setminus (\partial x \cup \partial y))} \right)^{-1} \\
&= 1 + \sum_{y \in U_{Px} \cap \partial x}\lambda_y \left( \frac{Z(U_{Pxy})}{Z(U_{Pxy} \setminus  \partial y)} \right)^{-1}. 
\end{align*} 
By induction, we may assume that $\Re\left( \frac{Z(U_{Pxy})}{Z(U_{Pxy} \setminus  \partial y)} \right) \geq 1/2$ and so by our choice of $\boldsymbol{\lambda}$, Lemma~\ref{lem:invariant} says that $\Re(\frac{Z(U_{Px})}{Z(U_{Px} \setminus \partial x)}) \geq 1/2$ proving (ii). This also immediately implies (i) (the second part of which was established earlier).
\end{proof}

\begin{proof}[Proof of Theorem~\ref{thm:almostclawfree}]
This follows from Lemma~\ref{lem:almostclawfree} by taking $x=u_K$, and hence $V(P)=\{x\}$ and $U_{Px}=V(G)$.
\end{proof}


\section{Examples} 
\label{sec:examples}

In this section we give examples showing the various ways in which it is not possible to extend the results we have presented so far. We will consider examples from larger and larger graph classes:


\begin{align*}
\text{line graph}&\subset\text{line graph of a multigraph}\\
&\subset\text{claw-free graph with a simplicial clique}\\
&\subset\text{claw-free graph}\\
&\subset\text{E-free graph, i.e., $S_{1,2,2}$-free}.
\end{align*}

\subsection{Even cycles}

Let $C_{2n}$ denote the even cycle with vertex set $V(C_{2n})=[2n]=\{0,1,\ldots,2n-1\}$ and edge set $E(C_{2n})=\{ij:j=i+1\pmod{2n}\}$. 
Many of our examples will be based on (even) cycles so we need to locate some of the multivariate zeros of $Z_{C_{2n},\bflambda}$. Since $C_{2n}$ is a line graph,
Theorem~\ref{thm:LRSmultivariate} with $k=2$ asserts that $R(1)$ is a multivariate zero-free region.  We show that this region cannot be extended further.  Denote by $B(z,\epsilon)$ the disc of radius $\epsilon$ about~$z$.



\begin{lemma}\label{lem:C2nnegative}
Suppose $z_0\in\partial R(1)$ and $\epsilon>0$.  Then there exist $n\in\mathbb{N}$, $z\in B(z_0,\epsilon)$ and $\bflambda:V(C_{2n})\to\{z,\zbar\}$ such that $Z_{C_{2n},\bflambda}=0$.
\end{lemma}

\begin{proof} 
Define $\bflambda:V(C_{2n})\to\mathbb{C}$ by 
$$
\lambda_i=\begin{cases}
\lambda,&\text{if $i$ is even;}\\ 
\mu,&\text{if $i$ is odd.}
\end{cases}
$$
A convenient way of keeping track of the calculation of $Z_{C_{2n,\bflambda}}$ is using a `transfer matrix'.  Consider a path $P_2$ of length~2 whose vertices, taken in order, have weights 1, $\mu$ and $\lambda$.  Define
$$
M=\begin{pmatrix}m_{00}&m_{01}\\m_{10}&m_{11}\end{pmatrix}=\begin{pmatrix}1&0\\0&\lambda\end{pmatrix} \begin{pmatrix}1&1\\1&0\end{pmatrix} \begin{pmatrix}1&0\\0&\mu\end{pmatrix} \begin{pmatrix}1&1\\1&0\end{pmatrix} =\begin{pmatrix}\mu+1&1\\\lambda&\lambda\end{pmatrix}.
$$
Note that $Z_{P_2,(1,\mu,\lambda)}=m_{00}+m_{01}+m_{10}+m_{11}$, and that $m_{00}$ is the contribution to the partition function from independent sets excluding both ends of $P_2$, that $m_{01}$ is the contribution of independent sets including the left endpoint of $P_2$ and excluding the right, and so on.  The transfer matrix for a path of length $2n$ with weights $(1,\mu,\lambda,\ldots,\mu,\lambda)$ is then simply $M^n$. Finally, the partition function for a cycle of length $2n$, being the path of length $2n$ with periodic boundary conditions, is $$Z_{C_{2n},\bflambda}=\Tr(M^n).$$
Write $M=P^{-1}\Lambda P$ with $\Lambda$ diagonal.  Then the partition function of the cycle may be written as 
$$
Z_{C_{2n},\bflambda}=\Tr(P^{-1}\Lambda^n P)=\Tr(\Lambda^n)=\alpha^n+\beta^n,
$$ 
where $\alpha,\beta$ are the eigenvalues of $\Lambda$. Note that to obtain a zero of the partition function we require $|\alpha|=|\beta|$ and that $|\arg(\alpha)-\arg(\beta)|$ divides $(2k+1)\pi$ for some $k\in\Nset$.

For simplicity assume $\mu=\lambdabar$.  The eigenvalues of $M$ are
$$
\alpha,\beta=\frac12\Big(\lambda+\mu+1\pm\sqrt{(\lambda+\mu+1)^2-4\lambda\mu}\Big)=
\frac12\Big(2\Re\lambda+1\pm\sqrt{(2\Re\lambda+1)^2-4|\lambda|^2}\Big).
$$
To ensure $|\alpha|=|\beta|$, we insist that 
\begin{equation}\label{eq:moduli}
4|\lambda|^2>(2\Re\lambda+1)^2;  
\end{equation}
and to ensure $\alpha^n+\beta^n=0$ we take 
\begin{equation}\label{eq:args}
\arg\alpha=\pi/2n=-\arg\beta. 
\end{equation}
 Let $\lambda=a+b\ii$, so that $\mu=a-b\ii$.  Then \eqref{eq:moduli} simplifies to
 \begin{equation}\label{eq:abcondition}
 4b^2>4a+1,
 \end{equation}
 and \eqref{eq:args} to 
 $$
 \tan\Big(\frac\pi{2n}\Big)=\frac{\sqrt{4b^2-4a-1}}{2a+1}
 $$
or
 \begin{equation}\label{eq:cycleroot}
 4b^2=\Big((2a+1)\tan\Big(\frac\pi{2n}\Big)\Big)^2+4a+1.
 \end{equation}
 Now, to find a pair of roots close to the parabola $R(1)$ with a desired real part~$a$, simply fix $a$ and take $n$ large.  As $n\to\infty$, the imaginary part $b$ tends to $\sqrt{a+\frac14}$ from above.  
 \end{proof}

Summarising, $R(1)$ is a multivariate zero-free region for even length cycles, and is the unique maximal such region that is symmetric around the real axis in the following sense: for any point $z$ on the boundary of $R(1)$, and $\epsilon>0$, the set $R(1)\cup B(z,\epsilon)\cup B(\bar z,\epsilon)$ is not multivariate zero free.  Note that there is no reason a priori to suppose that a unique maximal zero-free region exists, and indeed we shall encounter counterexamples later.

\subsection{Blow-ups of cycles (dense)}
Here we generalise the previous example to graphs of higher degree (and larger cliques) to show that Theorem~\ref{thm:almostclawfree} gives a zero-free region that is close to best possible.

By $C_{2n}[K_s]$ we denote the lexicographic product of a cycle of length $2n$ and a clique of size~$s\geq2$.  Thus, the vertex set of $C_{2n}[K_s]$ is $[2n]\times[s]$ and each vertex subset $\{i,i+1\}\times[s]$ (with arithmetic modulo $2n$) induces a $2s$-clique $K_{2s}$.  There are no further edges.  It is easy to check that $C_{2n}[K_s]$ is claw free;  indeed, $G[K_s]$ is claw free for any claw-free graph~$G$ and integer $s\geq2$.

Although $C_{2n}[K_s]$ is not a line graph (of a simple graph), it \textit{is} the line graph of a multigraph and so is amenable to the method of Asano contractions (see Theorem~\ref{thm:LRSmultivariate} and appendix).  As the cliques that comprise $C_{2n}[K_s]$ have size $2s$, this approach yields $R(s^2)$ as a multivariate zero-free region.  (Refer to Theorem~\ref{thm:LRSmultivariate}.) Alternatively, since any maximum clique in $C_{2n}[K_s]$ has size $2s$ and is simplicial, our Theorem~\ref{thm:almostclawfree} yields $R(2s)$ as a zero-free region.  
The result below shows that Theorem~\ref{thm:almostclawfree} is within a factor of $2$ of being best possible (i.e.\ we cannot replace $R(k)$ with $R(k')$ in the statement if $k'< k/2$).

\begin{lemma}
$R(s)$ is a multivariate zero-free region for the graph $C_{2n}[K_s]$, for all $s\geq2$.  It is the maximal such region that is symmetric about the real axis.     
\end{lemma}

\begin{proof}
Let $\bflambda:[2n]\times[s]\to \Cset$ be a labelling of the vertices of $C_{2n}[K_s]$.  Define $\bflambda':[2n]\to\Cset$ by $\lambda'_i=\sum_{j\in[s]}\lambda_{ij}$, for $i\in[2n]$.  We view $\bflambda'$ as a labelling of $C_{2n}$.  It is easy to verify that
$Z_{C_{2n}[K_s],\bflambda}=Z_{C_{2n},\bflambda'}$.  (Regard two independent sets in $C_{2n}[K_s]$ as equivalent if they intersect the same collection of sets of the form $\{i\}\times[s]$; now collect together contributions to the partition function from equivalent independent sets.)  Noting that $R(1)$ is a zero-free region for $C_{2n}$, we see that to establish the first part of the lemma it is enough to show that $\lambda_{ij}\in R(s)$ for all $j\in[s]$ implies $\lambda'_i\in R(1)$.

Fix $i\in[2n]$.  Since $\lambda_{ij}\in R(s)$ for all $j\in[s]$ we have 
$$
(\Im\lambda_{ij})^2<\frac1s\Re\lambda_{ij}+\frac1{4s^2},\quad\text{for all $j\in[s]$}.
$$
Summing over $j$, 
\begin{equation}\label{eq:R(1)}
\sum_{j\in[s]}(\Im\lambda_{ij})^2<\frac1s\Re\lambda'_{i}+\frac1{4s}.
\end{equation}
By Cauchy-Schwarz,
$$
(\Im\lambda'_i)^2
=\Bigg(\sum_{j\in[s]}1\cdot\Im\lambda_{ij}\Bigg)^2\leq \sum_{j\in[s]}1^2\sum_{j\in[s]}(\Im\lambda_{ij})^2=s\sum_{j\in[s]}(\Im\lambda_{ij})^2.
$$
Combining this with \eqref{eq:R(1)} we see that $\lambda'_i\in R(1)$. 

Optimality of the zero-free region $R(s)$ may be argued as follows.  Suppose $z_0\in\partial R(s)$ and $\epsilon>0$.  Note that $sz_0\in\partial R(1)$.  By Lemma~\ref{lem:C2nnegative}, there exist $z\in B(sz_0,s\epsilon)$ and $\bflambda':[2n]\to\{z,\zbar\}$ such that $Z_{C_{2n},\bflambda'}=0$.  Now let $\bflambda=\bflambda'/s$. Then $z/s\in B(z_0,\epsilon)$, $\bflambda:[2n]\to\{z/s,\zbar/s\}$ and $Z_{C_{2n}[K_s],\bflambda}
=Z_{C_{2n},\bflambda'}=0$. 
\end{proof}

\subsection{Paths with longer range interactions}
The following are examples of graphs in $\mathcal{G}^{\textup{Cl-S}}$ that are not line graphs thus showing that Theorem~\ref{thm:almostclawfree} applies to graphs that Theorem~\ref{thm:LRSmultivariate} does not apply to.

For $n,d\geq1$, let $P_n^{(d)}$ denote the path of length $n$ with adjacencies up to distance~$d$.  So $V(P_n^{(d)})=[n]$ and $E(P_n^{(d)})=\{ij:i,j\in[n]\text{ and }1\leq j-i\leq d\}$.  Assume that $d\geq2$. Then for $n$ sufficiently large, the graph $P_n^{(d)}$ is not the line graph of a multigraph. (As the property is hereditary, we can take $n=3d$ without loss of generality. Any line graph can be covered by cliques in such a way that each vertex is in at most two cliques.
Assume that the edges of $P_{3d}^{(d)}$ could be covered by such cliques. In particular,  the edge $\{d,2d\}$ must be in some such clique~$K$.  The vertex $d$ has degree $2d$, so to cover all its incident edges using two cliques requires $K$ to be of maximum cardinality, i.e., $K=\{d,d+1,\ldots,2d\}$.  This in turn forces us to take the cliques $\{0,\ldots,d\}$, $\{d,\ldots,2d\}$ and $\{2d,\ldots,3d\}$.  But now it is impossible to cover (say) the edges $\{1,d+1\}$ and $\{d+1,2d+1\}$ using a single extra clique.)  Thus, this example is not amenable to the Asano contraction approach.

However, $P_n^{(d)}$ is claw free and has a simplicial clique, for example, the singleton vertex $\{0\}$.  Its maximum clique size is $d+1$, so by Theorem~\ref{thm:almostclawfree} it has $R(d+1)$ as a multivariate zero-free region.  

\subsection{Blow-ups of cycles (sparse)}
By $C_{2n}[sK_1]$ we denote the lexicographic product of a cycle of length $2n$ and an independent set of size~$s\geq2$.  Thus, the vertex set of $C_{2n}[sK_1]$ is $[2n]\times[s]$ and each pair of vertex subsets $\{i\}\times[s]$ and $\{i+1\}\times[s]$ (with arithmetic modulo $2n$) induces a complete bipartite graph $K_{s,s}$.  There are no further edges.  It is easy to check that $C_{2n}[sK_1]$ is E-free (and hence $S_{i,j,k}$-free for $i \geq 1$ and $j,k \geq 2$); indeed, for any E-free graph~$G$ and integer $s\geq2$, the graph $G[sK_1]$ is also E-free   

Since $C_{2n}[2K_1]$ is E-free, we deduce from Theorem~\ref{thm:zerofreeinterval} the existence of an open multivariate zero-free region containing any finite segment of the positive real axis.  We cannot, however, deduce the existence of an open zero-free region containing the entire real axis. In fact, we now show that no such region exists.

\begin{lemma}
Consider the collection $\mathcal{C}=\{C_{2n}[2K_1]:n\in\Nset\}$ of blown-up cycles defined above.  There is no open multivariate zero-free region for $\mathcal C$ that contains the positive real axis.
\end{lemma}

\begin{proof}
Given constants $\epsilon>0$ and $W>0$,  define a weighting $\bflambda^{\epsilon, W}:[2n]\times[2]\to\Cset$ of the vertices of $C_{2n}[2K_1]$ by
$$
\lambda_{ij}^{\epsilon, W}=\begin{cases}
    W,&\text{if $j=0$;}\\
    w_\epsilon=1+\epsilon\ii,&\text{if $i$ is even and $j=1$};\\
    \bar{w}_\epsilon=1-\epsilon\ii,&\text{if $i$ is odd and $j=1$}.
\end{cases}
$$
  As before, by grouping equivalent independent sets, we can write down an equivalent (in the sense of preserving the partition function) weighting $\bfmu^{\epsilon, W}$ of the simple cycle $C_{2n}$:
$$
\mu^{\epsilon, W}_i=\begin{cases}
    z_{\epsilon, W} := (W+1)(w_\epsilon+1)-1=(2W+1)+(W+1)\epsilon\ii,&\text{if $i$ is even;}\\
    \bar{z}_{\epsilon, W} := (W+1)(\bar{w}_\epsilon+1)-1=(2W+1)-(W+1)\epsilon\ii,&\text{if $i$ is odd.}
\end{cases}
$$
We then have $Z_{C_{2n}[2K_1],\bflambda^{\epsilon, W}}=Z_{C_{2n},\bfmu^{\epsilon, W}}$.

With a view to obtaining a contradiction, assume that $S$ is an open multivariate zero-free region for $\mathcal{C}$ that contains the positive real axis.  Choose $\delta>0$ such that the interval $[1-\delta\ii,1+\delta\ii]$ is contained in~$S$. We will check that there exists $W>0$, $ \epsilon \in (0, \delta]$, and $n \in \mathbb{N}$ such that $Z_{C_{2n},\bfmu^{\epsilon, W}} = 0$ and hence $Z_{C_{2n}[2K_1],\bflambda^{\epsilon, W}}$ giving the contradiction.

Consider the set
\begin{align*}
T_{\delta} = \{ z_{\epsilon, W}: \epsilon \in (0, \delta], \; W>0 \} 
&= \{(2W+1)+(W+1)\epsilon\ii: \; \epsilon \in (0,  \delta], \; W>0\} \\
&= \{x+ \ii y: 0 < y < \tfrac{1}{2} \delta(x+1), \; x>1 \}.
\end{align*}
It is easy to check that $T_{\delta}$ contains a (unbounded) part of the parabola $\partial R(1)$. Lemma~\ref{lem:C2nnegative} then implies that there exists $z_{\epsilon, W} \in T_{\delta}$ and $n \in \mathbb{N}$ such that alternately weighting the vertices of $C_{2n}$ with $z_{\epsilon, W}$ and $\bar{z}_{\epsilon, W}$ gives a zero of $Z_{C_{2n}}$, i.e.\ $Z_{C_{2n},\bfmu^{\epsilon, W}} = 0$.   
\end{proof}

\subsection{Complete multipartite graphs}
In this subsection, we show that we cannot drop the bounded degree condition in Theorem~\ref{thm:univariate} and Theorem~\ref{thm:zerofreeinterval}.
For $a,b,n,m \in \mathbb{N}$, let $K(a,b;n,m)$ be the graph that has $an + bm$ vertices which are divided into $a$ sets with $n$ vertices and a further $b$ sets of $m$ vertices and where all edges are present between vertices in different sets and no edges are present inside any of the sets. So this is a complete multipartite graph. Note that $K(a,b;n,m)$ is $S_{i,j,k}$-free for any choice of positive integers $i,j,k$ except for $i=j=k=1$, i.e., the graph may contain claws but does not contain any other subdivided claw.

One can easily verify that the (univariate) independence polynomial of $K(a,b; n,m)$ is given by 
\[
Z_{K(a,b;n,m), \lambda} = a[(1+\lambda)^n - 1] + b[(1+ \lambda)^m - 1)] + 1 = a(1+\lambda)^n + b(1+\lambda)^m + (1-a-b).
\]
From the following lemma, it immediately follows that the roots of such polynomials are dense in the complex plane except possibly inside a disc of radius $1$ centered at $-1$. This shows in particular that zeros accumulate on the positive real line and so we cannot expect a zero-free region around the positive real line (or any section of it) for the general class of $S_{i,j,k}$-free graphs where $i,j \geq 1$ and $k \geq 2$.

\begin{lemma}
For the set of polynomials of the form $ax^{2n} + bx^n + (1-a-b)$, where $a,b,n \in \mathbb{N}$, the roots are dense outside the unit disc. 
\end{lemma}
Note that there are zeros inside the unit disc, but we do not describe them here.
\begin{proof}
First consider quadratics of the form $ax^2 + bx + (1-a-b)$, where $a$ and $b$ are positive integers. We claim that the roots of such quadratics are dense in the interval $(-\infty, -1)$. Indeed, given any positive rational $t$, we can find a root arbitrarily close to $-1-t$ as follows. Take $L$ to be a very large positive integer so that $b=tL$ is an integer and set $a=L$. Then one of the roots of $ax^2 + bx + (1-a-b)$ is
\begin{align*}
-\frac{b}{2a} -  \frac{(b^2 - 4a(1-a-b))^{1/2}}{2a} 
= -\frac{t}{2} - \left( \left( \frac{t}{2}+1 \right)^2 - \frac{1}{L} \right)^{1/2} 
= -t-1 + O(L^{-1}).
\end{align*}
proving the claim.

Next, given any $z = re^{\ii\theta}$ with $r>1$ and $\varepsilon > 0$, consider a small neighbourhood of $z$ given by
\[
N_{\varepsilon}(z) = \{r'e^{\ii\theta'} : |r-r'| \leq \varepsilon, |\theta - \theta'| \leq \varepsilon/r \},
\]
which is easily checked to be contained inside a disc of radius $3 \varepsilon$ centered at $z$. For a sufficiently large positive integer $N$ (we can take $N \geq \pi r\varepsilon^{-1}$ ), the set $N_{\varepsilon}(z)^N$ (raising each element in $N_{\varepsilon}(z)$ to the power $N$) is an annulus between the circles of radius $(r-\varepsilon)^N$ and $(r+\varepsilon)^N$ and in particular crosses the interval $(-\infty, -1)$ and so contains a root of $Ax^2 + Bx + (1-A-B)$ for some positive integers $A,B$. Therefore, some complex number within a distance $3\varepsilon$ of $z$ is a root of $Ax^{2N} + Bx^N + (1 -A - B)$.
\end{proof}

\section{The case when $H$ is neither a subdivided claw nor a path}
\label{sec:notclaw}


Recall that Theorem~\ref{thm:univariate} and Theorem~\ref{thm:zerofreeinterval} give a zero-free region for the independence polynomials of $H$-free graphs of fixed maximum degree $\Delta$, where $H$ is a fixed subdivided claw $S_{i,j,k}$. Note that similar results hold when $H = P_k$ is a path of some fixed length~$k$ for the trivial reason that there are only finitely many $P_k$-free graphs of maximum degree~$\Delta$.
In this section, we show that no similar results can hold for any graph~$H$ that is not a subdivided claw or a path. 

Given a graph $H$ which is neither a path nor a subdivided claw, then either $H$ has a vertex of degree at least $4$ or $H$ has (at least) two vertices of degree $3$.
In the latter case, set $k\in\Nset$ to be larger than the distance between these vertices of degree $3$ and in the former case, set $k=0$.
Let $\mathcal T=\{T_d:d\in\Nset\}$ be the family of trees, where $T_d$ is obtained from the complete binary tree of height~$d$ by subdividing each edge $2k$ times (so there are $2k$ intermediate vertices between branch vertices). Refer to Figure~\ref{fig:trees}.  By our choice of $k$, all trees in $\mathcal{T}$ are $H$-free. 

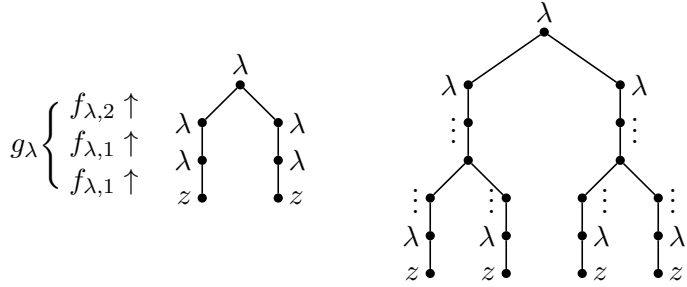
\begin{figure}
\begin{tikzpicture}[scale=1, semithick]

   \draw (-3,1) node [vertex] (k1) {} node [left] {$z$};
   \draw (-2,1) node [vertex] (k2) {} node [right] {$z$};
   \draw (-3,1.5) node [vertex] (j1) {} node [left] {$\lambda$};
   \draw (-2,1.5) node [vertex] (j2) {} node [right] {$\lambda$};
   \draw (-3,2) node [vertex] (i1) {} node [left] {$\lambda$};
   \draw (-2,2) node [vertex] (i2) {} node [right] {$\lambda$};
   \draw (-2.5,2.5) node [vertex] (h) {} node [above] {$\lambda$};
   \draw (k1) -- (j1) -- (i1) -- (h);
   \draw (k2) -- (j2) -- (i2) -- (h);
   \draw (-3.7,1.2) node [left] {$f_{\lambda,1}\uparrow$};
   \draw (-3.7,1.7) node [left] {$f_{\lambda,1}\uparrow$};
   \draw (-3.7,2.2) node [left] {$f_{\lambda,2}\uparrow$};
   \draw (-4.7,1.7) node [left] {$g_\lambda\Bigg\{$};
      
   \draw (0,0) node [vertex] (g1) {} node [left] {$z$};
   \draw (1,0) node [vertex] (g2) {} node [left] {$z$};
   \draw (0,0.5) node [vertex] (f1) {} node [left] {$\lambda$};
   \draw (1,0.5) node [vertex] (f2) {} node [left] {$\lambda$};
   \draw (0,1) node [vertex] (e1) {} node [left] {\strut$\vdots$};
   \draw (1,1) node [vertex] (e2) {} node [left] {\strut$\vdots$};
   \draw (0.5,1.5) node [vertex] (d1) {};
   \draw (g1) -- (f1) -- (e1) -- (d1);
   \draw (g2) -- (f2) -- (e2) -- (d1);
   
   \draw (2,0) node [vertex] (g3) {} node [right] {$z$};
   \draw (3,0) node [vertex] (g4) {} node [right] {$z$};
   \draw (2,0.5) node [vertex] (f3) {} node [right] {$\lambda$};
   \draw (3,0.5) node [vertex] (f4) {} node [right] {$\lambda$};
   \draw (2,1) node [vertex] (e3) {} node [right] {\strut\;$\vdots$};
   \draw (3,1) node [vertex] (e4) {}  node [right] {\strut$\vdots$};
   \draw (2.5,1.5) node [vertex] (d2) {};
   \draw (g3) -- (f3) -- (e3) -- (d2);
   \draw (g4) -- (f4) -- (e4) -- (d2);

   \draw (0.5,2) node [vertex] (c1) {} node [left] {$\vdots$};
   \draw (2.5,2) node [vertex] (c2) {} node [right] {$\vdots$};
   \draw (0.5,2.5) node [vertex] (b1) {} node [left] {$\lambda$};
   \draw (2.5,2.5) node [vertex] (b2) {} node [right] {$\lambda$};
   \draw (1.5,3.2) node [vertex] (a) {} node [above] {$\lambda$};
   \draw (d1) -- (c1) -- (b1) -- (a);
   \draw (d2) -- (c2) -- (b2) -- (a);

\end{tikzpicture}
\caption{The trees $T_1$ and  $T_2$, when $k=1$.}
\label{fig:trees}
\end{figure}



\begin{lemma}\label{lem:notclaw}
The set of zeros
$$
\big\{\lambda\in\Cset: Z_{T,\lambda}=0\text{ for some }T\in\mathcal T\big\}
$$
has an accumulation point on the positive real axis.
\end{lemma}

The lemma shows that, for any choice of $H$ that is not a subdivided claw or a path, there can be no zero-free region containing the positive real line for independence polynomials of $H$-free graphs of maximum degree $3$. Our proof closely follows the treatment of Buys~\cite{Buys} and so we give details where our proof differs from~\cite{Buys} and only outline the rest. 

\begin{proof}[Proof of Lemma \ref{lem:notclaw}]
The advantage of working with trees is that we have an inductive description of the partition function.  A particularly convenient way of managing the calculation is through the occupation ratio of the root of a tree:  the total weight of independent sets that include the root vertex divided by the total weight of independent sets that exclude it.  The recurrence for the occupation ratio was introduced into this area by Weitz~\cite{Weitz}, and can be found in earlier work, such as Kelly~\cite{Kelly}.  The derivation of the recurrence can be found in Buys~\cite{Buys}. 
Define 
$$
f_{\lambda,1}(z)=\frac\lambda{1+z},\> f_{\lambda,2}(z)=\frac\lambda{(1+z)^2}\>\text{ and }\> g_\lambda(z)=f_{\lambda,2}\circ f_{\lambda,1}^{\circ 2k}(z)=f_{\lambda,2}\circ \underbrace{f_{\lambda,1}\circ\cdots\circ f_{\lambda,1}}_\text{$2k$ copies}(z).
$$
The functions $f_{\lambda,1}$ and $f_{\lambda,2}$ express how the occupation ratio is transformed as we go one level up in the tree;  $f_{\lambda,1}$ (respectively, $f_{\lambda,2}$) relates to vertices with branching factor~1 (respectively, 2).  Refer to Figure~\ref{fig:trees}.  It follows that $g_\lambda^{\circ d}(\lambda)$ is the occupation ratio of the root of~$T_d$.   If $g_\lambda^{\circ d}(\lambda)=-1$ then the contributions from independent sets that include the root and those that exclude it exactly cancel out and we have $Z_{T_d,\lambda}=0$.  Following Buys, our initial  goal is to find a fugacity $\lambda_0\in\Rset_{>0}$ such that $g_{\lambda_0}$ has an indifferent fixed point~$z_{\lambda_0}$, that is to say $g_{\lambda_0}(z_{\lambda_0})=z_{\lambda_0}$ and $|g'_{\lambda_0}(z_{\lambda_0})|=1$.

Fix $\lambda\in\Rset_{>0}$.
Observe that $f_{\lambda,1}$ and $f_{\lambda,2}$ are monotonically decreasing as functions on $\Rset_{\geq0}$, and hence so is $g_\lambda$, being a composition of an odd number of these. As $g_\lambda$ is also non-negative on $\Rset_{\geq0}$ it intersects the identity function at a unique point in~$\Rset_{\geq0}$; call this point $z_\lambda$.  Since $g_0$ is identically zero, we have $z_0=0$ and $|g'_0(z_0)|=0$ . We will show that $|g'_\lambda(z_\lambda)|\to2$ as $\lambda\to\infty$ from which it follows, by continuity,\footnote{The fixed point $z=z_\lambda$ is defined by the implicit equation $t(\lambda,z)=0$, where $t(\lambda,z)=g_\lambda(z)-z$. Note that $\frac{\partial t}{\partial z}$ is at most $-1$ and, in particular is non-zero.  Now apply the implicit function theorem.} 
that there exists at least one $\lambda\in\Rset_{>0}$ at which $|g'_\lambda(z_\lambda)|=1$.  This will give us the sought for indifferent fixed point.

As $k$ is constant, and we are taking $\lambda$ to be sufficiently large, we can approximate the occupation ratio of the root of~$T_1$ --- with the two leaves weighted~$z$, and the remaining vertices weighted~$\lambda$ --- by enumerating just the maximum independent sets.  Then the occupation ratio of the root is $g_\lambda(z)=w_\mathrm{in}/w_\mathrm{out}$, where 
\begin{align*}
w_\mathrm{in}&=\big(\lambda^{2k+1}+O(\lambda^{2k})\big)+2z\big(k\lambda^{2k}+O(\lambda^{2k-1})\big)+z^2\big(k^2\lambda^{2k-1}+O(\lambda^{2k-2})\big)\\
&=\lambda^{2k-1}(\lambda+kz)^2(1+ O(\lambda^{-1})),\\
\noalign{\noindent and}
w_\mathrm{out}&=\big((k+1)^2\lambda^{2k}+O(\lambda^{2k-1})\big)+2z\big((k+1)\lambda^{2k}+O(\lambda^{2k-1})\big)+z^2\big(\lambda^{2k}+O(\lambda^{2k-1})\big)\\
&=\lambda^{2k}(z+k+1)^2(1+ O(\lambda^{-1})).
\end{align*}
Thus
$$
g_\lambda(z)=\frac{(\lambda+kz)^2}{\lambda(z+k+1)^2}(1\pm O(\lambda^{-1})).
$$

With a view to verifying that $g_\lambda$ has a fixed point $z_\lambda$ close to $\lambda^{1/3}$, let $\epsilon>0$ and define $z^\pm_\lambda=(1\pm\epsilon)\lambda^{1/3}$.   Then
\begin{align*}
g_\lambda(z^+_\lambda)&=\frac{(\lambda+ O(\lambda^{1/3}))^2(1+ O(\lambda^{-1}))}{\lambda\big((1+\epsilon)\lambda^{1/3}+ O(1)\big)^2}\\
&=\frac{\lambda^2(1+ O(\lambda^{-2/3}))}{(1+\epsilon)^2\lambda^{5/3}(1+ O(\lambda^{-1/3}))}\\
&=(1+\epsilon)^{-2}\lambda^{1/3}(1+ O(\lambda^{-1/3})).
\end{align*} 
Thus, provided $\lambda\gg\epsilon^{-3}$, we have $z^+_\lambda-g_\lambda(z_\lambda^+)>0$ and $z^-_\lambda-g_\lambda(z_\lambda^-)<0$.  It follows that $g_\lambda$ has a fixed point $z_\lambda$ satisfying $z_\lambda^-<z_\lambda<z_\lambda^+$, at which $z_\lambda=\lambda^{1/3}(1+ O(\lambda^{-1/3}))=\lambda^{1/3}+ O(1)$.

Letting $\sim$ denote equality up to a factor $(1\pm o_\lambda(1))$, 
$$
f_{\lambda,1}^{\circ \ell}(z_\lambda)\sim\begin{cases}
\lambda^{1/3},&\text{if $\ell$ is even;}\\
\lambda^{2/3},&\text{if $\ell$ is odd.}
\end{cases}
$$
By the chain rule,
$$
g'_\lambda(z)
=(f'_{\lambda,2}\circ f_{\lambda,1}^{\circ 2k}(z))
\times (f'_{\lambda,1}\circ f_{\lambda,1}^{\circ (2k-1)}(z))
\times \cdots \times (f'_{\lambda,1}\circ f_{\lambda,1}(z))
\times f'_{\lambda,1}(z).
$$
Noting
$$
f'_{\lambda,1}(z)=\frac{-\lambda}{(1+z)^2}\> \text{ and }\> f'_{\lambda,2}(z)=\frac{-2\lambda}{(1+z)^3},
$$
we have 
$$
g'_\lambda(z_\lambda)\sim-2 \times \underbrace{(-\lambda^{-1/3} \times -\lambda^{1/3})\times\cdots\times (-\lambda^{-1/3} \times -\lambda^{1/3})}_{\text{$k$ factors}}=-2.
$$
By continuity, there exists $\lambda\in\Rset_{>0}$ such that $|g'_\lambda(z_\lambda)|=1$.    Let this $\lambda$ be $\lambda_0$.  Then $z_{\lambda_0}$ is an indifferent fixed point of $g_{\lambda_0}$.

We now employ \cite[Lemma 13]{Buys}.   In our application, $\Delta=2$, and $H_\Delta=H_2$ is the set of all functions that can be obtained by composing sequences composed of $f_{\lambda,1}$ and $f_{\lambda,2}$.  The lemma asserts that the existence of an indifferent fixed point at $\lambda=\lambda_0$ (just established) ensures that $\lambda_0$ is an accumulation point of zeros realised by balanced trees with branching factors in $\{1,2\}$.  In our application, however, we want to ensure that pairs of degree-2 vertices remain separated, by $2k$ degree-1 vertices.  In fact, the trees $T\in H_2$ referred to in the statement of Buys' Lemma~13 do have this property.  However, we have to peek into the proof to see this.

The proof of the lemma first establishes that a certain function $g_\lambda^{\circ n}\circ h_\lambda$ is not \emph{normal} about $\lambda=\lambda_0$ (or more precisely that the family $\{\lambda\mapsto g_{\lambda}^{n}\circ h_{\lambda}(0)\}_{n\ge1}$ is not normal in any neighbourhood of $\lambda_{0}$).  The function $g_\lambda$ and the point $\lambda_0$ are the same as ours, so the issue for us is whether the $h_\lambda\in H_2$ has two occurrences of $f_{\lambda,2}$ that are too close together.  Consulting Remark~8, we discover that $h_\lambda$ is in fact an initial segment of the composition $f_{\lambda,2}\circ f_{\lambda,1}^{\circ 2k}$, so in fact contains at most one occurrence of $f_{\lambda,2}$.  Finally, at the end of the proof of Lemma 13, either one or two functions are removed from the composition of functions $g_\lambda^{\circ n}\circ h_\lambda$ to yield the desired function $\tilde g_\lambda$.  Thus, $\tilde g_\lambda$ is a composition of a contiguous subsequence of functions from 
$$
g_{\lambda}^{\circ(n+1)}=(f_{\lambda,2}\circ f_{\lambda,1}\circ\cdots\circ f_{\lambda,1})\circ (f_{\lambda,2}\circ f_{\lambda,1}\circ\cdots\circ f_{\lambda,1})\circ\cdots\circ(f_{\lambda,2}\circ f_{\lambda,1}\circ\cdots\circ f_{\lambda,1}).
$$
The corresponding tree has pairs of degree-2 vertices separated by sequences of $2k$ degree-1 vertices, as desired.
\end{proof}

\section{Concluding Remarks}

A few natural questions arise from this work. First, in Theorem~\ref{thm:univariate}, we establish an open, zero-free region $F$ containing $[0, \infty)$  for $S_{t,t,t}$-free graphs of maximum degree $\Delta$. Can we in fact take $F$ to be a sector of the complex plane with a suitably small angle (depending on $\Delta$ and $t$)?

Secondly, it would be interesting to know whether some form of Theorem~\ref{thm:almostclawfree} can be extended to all claw-free graphs of bounded degree (or bounded clique size). A more basic question is whether there is an open multivariate zero-free region containing $[0, \infty)$ for claw-free graphs of bounded degree (or bounded clique size).

\section*{Acknowledgments}
We thank Guus Regts and Ferenc Bencs for helpful discussions and feedback on an earlier draft. We also thank the anonymous referee for their careful reading and valuable feedback on this work.

\bibliographystyle{plain}
\bibliography{HardCoreZeros}

\section{Appendix}
\label{sec:appendix}

The line graph of a multigraph may be characterised as a graph~$G$ for which there exist vertex subsets $V_1,V_2,\ldots V_\ell\subseteq V(G)$ with the following properties: (i)~every induced graph $G[V_i]$ is a clique; (ii)~every edge $e\in E(G)$ is contained in some clique $G[V_i]$; and (iii)~every vertex $v\in V(G)$ is contained in at most two cliques.  

\begin{proof}[Proof of Theorem~\ref{thm:LRSmultivariate}]
Suppose $G$ is a line graph of a multigraph, whose maximum clique in the above decomposition has size $k_0$.  We show that $R(k_0^2/4)$ is a multivariate zero-free region for $G$.  
Note that $k_0$ is at most the size of the largest clique in~$G$, so this shows a little more than claimed in Theorem~\ref{thm:LRSmultivariate}.

We use the technique of Asano contractions, as employed in this context by Lebowitz, Ruelle and Speer.  The proof of our claim is essentially contained in the terse proof of their Lemma~2.6~\cite{LebowitzRuelleSpeer}.  Here, we sketch the line of proof, specialise it to our specific application, and explicitly compute the parabolic region.

The independence polynomial of the clique $K_k$ is simply $z_1+z_2+\cdots z_k+1$, and has $\{z:\Re z>-k^{-1}\}$ as a multivariate zero-free region.  View the line graph~$G$ as being constructed from a set of disjoint cliques by sequentially identifying pairs of vertices.  Each clique individually has the halfspace $\Phi=\{z:\Re z>-k_0^{-1}\}$ as a multivariate zero-free region.  Clearly, the disjoint union of cliques $G_0=G[V_1]\cupdot G[V_2]\cupdot\cdots\cupdot G[V_\ell]$ also has $\Phi$ as a multivariate zero-free region. A halfspace is a special case of a ``circular region'', so Asano contraction is applicable.  

Let the independence polynomial of $G_0$ be $p(z_1,\ldots,z_N)$.  Think of the zero-free region as being the cartesian product $\Phi_1\times\Phi_2\times\cdot\times\Phi_N$, where, initially, $\Phi_i=\Phi$ for all~$i$.  If $z_i\in\Phi_i$ for all~$i$ then $p(z_1,\ldots,z_N)\not=0$.  We can recover $G$ from~$G_0$ by repeatedly identifying pairs of vertices.  Suppose, without loss of generality, that we identify the vertices labelled by $z_{N-1}$ and~$z_N$ first, and assign variable~$z$ to the coalesced vertex.  The key contraction lemma (see \cite[Lemma~2.4]{LebowitzRuelleSpeer}) says that the resulting independence polynomial, in variables $(z_1,\ldots,z_{N-2},z)$, has a zero-free region $\Phi_1\times\Phi_2\times\cdot\times\Phi_{N-2}\times\Phi'$, where $\Phi_1,\Phi_2\ldots,\Phi_{N-2}=\Phi$ and $\Phi'$ is obtained from $\Phi_{N-1}$ and $\Phi_N$ according to a certain rule, namely
\begin{align*}
\Phi'=\overline{-\overline{\Phi_{N-1}}\cdot\overline{\Phi_N}}&=\overline{\{-w_1w_2:\text{$w_1\notin\Phi_{N-1}$ and $w_2\notin\Phi_N$}\}}\\&=\overline{\{-w_1w_2:\Re w_1,\Re w_2\leq-k_0^{-1}\}},
\end{align*}
where overline denotes complement of a set.  

Then, repeatedly applying Asano contractions we reach the target graph~$G$, and discover that it has $(\Phi')^n$ as a zero free set, where $n=|V(G)|$.  (It is here that we crucially use the fact that the original graph $G$ is a line graph.  As each vertex is contained in at most two cliques, we are always replacing two zero-free sets~$\Phi$ by a single~$\Phi'$.  If $G$ is the line graph of a \textit{multi}graph, then $G$ may contain parallel edges, but this is fortunately not an issue for the hard-core model.)  

It only remains to show that $\Phi'=R(k_0^2/4)$.  The set $\Phi'$ is open and we want to identify its boundary $\partial\Phi'$.  If either $\Re w_1<-k_0^{-1}$ or $\Re w_2<-k_0^{-1}$ then by making small perturbations of $w_1$ or $w_2$ (note that neither is zero) within $\overline\Phi$ we can move in a small open region around $w_1w_2$.  So, the boundary of $\Phi'$ must be defined by $w_1, w_2$ lying on the boundary of~$\Phi$.  Let $w_1=-k_0^{-1}+y_1\ii$ and $w_2=-k_0^{-1}+y_2\ii$, with $y_1,y_2\in\Rset$. Then 
$$
-w_1w_2=f(y_1,y_2)=(y_1y_2-k_0^{-2})+k_0^{-1}(y_1+y_2)\,\ii.
$$
Now,
$\partial f/\partial y_1= y_2+k_0^{-1}\ii$ and $\partial f/\partial y_2= y_1+k_0^{-1}\ii$, so unless $y_1=y_2$ we can again move in an open set around $f(y_1,y_2)$, and hence $f(y_1,y_2)$ cannot be on the boundary of $\Phi'$.  Setting $y=y_1=y_2$ we obtain the parametric expression $(y^2-k_0^{-2})+2k_0^{-1}y\,\ii$ for the boundary of $R(k_0^2/4)$ and hence $\Phi' = R(k_0^2/4)$.
\end{proof}

\end{document}